\documentclass[a4paper,12pt]{article}
\usepackage{amsfonts}
\usepackage{mathrsfs}
\usepackage{algorithm, algpseudocode}
\usepackage{amsmath,amssymb,amsthm,latexsym,amsfonts}
\usepackage{pstricks}
\usepackage{enumerate}

\usepackage{graphicx}
\usepackage{color}
\usepackage{ifpdf}

\usepackage{caption}

\begin{document}

\newtheorem{lem}{Lemma}
\newtheorem{thm}{Theorem}
\newtheorem{cor}{Corollary}
\newtheorem{exa}{Example}
\newtheorem{con}{Conjecture}
\newtheorem{rem}{Remark}
\newtheorem{obs}{Observation}
\newtheorem{definition}{Definition}
\newtheorem{pro}{Proposition}
\theoremstyle{plain}
\newcommand{\D}{\displaystyle}
\newcommand{\DF}[2]{\D\frac{#1}{#2}}

\renewcommand{\figurename}{{\bf Fig}}
\captionsetup{labelfont=bf}

\title{The $k$-proper index of complete bipartite and complete multipartite
graphs\footnote{Supported by NSFC No.11371205 and 11531011, and PCSIRT.}}
\author{\small Wenjing Li, Xueliang Li, Jingshu Zhang\\
      {\it\small Center for Combinatorics and LPMC}\\
      {\it\small Nankai University, Tianjin 300071, China}\\
       {\it\small liwenjing610@mail.nankai.edu.cn; lxl@nankai.edu.cn; jszhang@mail.nankai.edu.cn}}
\date{}
\maketitle

\begin{abstract}
Let $G$ be an edge-colored graph. A tree $T$ in $G$ is a \emph{proper tree} if no two adjacent edges of it are assigned the same color. Let $k$ be a fixed integer with $2\leq k\leq n$.
For a vertex subset $S\subseteq V(G)$ with $|S|\geq 2$, a tree is called an \emph{$S$-tree} if it connects $S$ in $G$ . A \emph{$k$-proper coloring} of $G$ is an edge-coloring of $G$ having the property that
for every set $S$ of $k$ vertices of $G$, there exists a proper $S$-tree $T$ in $G$.
The minimum number of colors that are required in a $k$-proper coloring of $G$
is defined as the \emph{$k$-proper index} of $G$, denoted by $px_k(G)$. In this paper, we determine the 3-proper index of all complete bipartite and complete multipartite graphs and partially determine the $k$-proper index of them for $k\geq 4$.\\[2mm]

\end{abstract}

\section{Introduction}
All graphs considered in this paper are simple, finite, undirected and connected.
We follow the terminology and notation of Bondy and Murty in~\cite{Bondy} for those not
defined here. Let $G$ be a graph, we use $V(G), E(G),|G|,\Delta(G)$ and $\delta(G)$ to denote the vertex set, edge set, order (number of vertices), maximum degree and minimum degree of $G$, respectively. For $D\subseteq V(G)$, let $\overline{D}=V(G)\backslash D$, and let $G[D]$ denote the subgraph of $G$ induced by $D$.

Let $G$ be a nontrivial connected graph with an \emph{edge-coloring c}
$:E(G)\rightarrow \{1,\dots,t\}$, $t\in \mathbb{N}$, where adjacent edges may
be colored with the same color. If adjacent edges of $G$ receive different colors by $c$, then $c$ is called a \emph{proper coloring}. The minimum number of colors required in a proper coloring of $G$ is referred as the \emph{chromatic index} of $G$ and denoted by $\chi'(G)$. Meanwhile, a path in $G$ is called \emph{a rainbow path} if no two edges
of the path are colored with the same color. The graph $G$ is called \emph{rainbow connected}
if for any two distinct vertices of $G$, there is a rainbow path connecting them.
For a connected graph $G$, the \emph{rainbow connection number} of $G$, denoted by $rc(G)$, is defined as the
minimum number of colors that are required to make $G$ rainbow connected.
These concepts
were first introduced by Chartrand et al. in~\cite{Char1} and have been well-studied since then.
For further details, we refer the reader to a book~\cite{Li}.


Motivated by rainbow coloring and proper coloring in graphs, Andrews et al.~\cite{Andrews} and, independently, Borozan et al.~\cite{Magnant} introduced the concept of proper-path coloring.
Let $G$ be a nontrivial connected graph with an edge-coloring.
A path in $G$ is called \emph{a proper path} if no two adjacent edges of the path are
colored with the same color. The graph $G$ is called \emph{proper connected}
if for any two distinct vertices of $G$,
there is a proper path connecting them.
The \emph{proper connection number} of $G$,
denoted by $pc(G)$, is defined as the minimum number of colors that are required to make $G$ proper connected.
For more details, we refer to a dynamic survey~\cite{Li1}.

Chen et al.~\cite{Chen1} recently generalized the concept of proper-path to proper tree.
A tree $T$ in an edge-colored graph is a \emph{proper tree} if no two adjacent edges of it are assigned the same color. For a vertex subset $S\subseteq V(G)$,
a tree is called an \emph{$S$-tree} if it connects $S$ in $G$.
Let $G$ be a connected graph of order $n$ with an edge-coloring and let $k$ be a fixed integer with $2\leq k\leq n$. A \emph{$k$-proper coloring} of $G$ is an edge-coloring of $G$ having the property that for every set $S$ of $k$ vertices of $G$, there exists a proper $S$-tree $T$ in $G$.
The minimum number of colors that are required in a $k$-proper coloring of $G$
is the \emph{$k$-proper index} of $G$, denoted by $px_k(G)$. Clearly, $px_2(G)$ is precisely the proper connection number $pc(G)$ of $G$.
For a connected graph $G$, it is easy to see that
$px_2(G)\leq px_3(G)\leq\dots\leq px_n(G)$. The following results are not difficult to get.

\begin{pro}\cite{Chen1}\label{pro1}
If $G$ is a nontrivial connected graph of order $n\geq 3$, and $H$ is a connected spanning subgraph of $G$, then $px_k(G)\leq px_k(H)$ for any $k$ with $3\leq k\leq n$. In particular, $px_k(G)\leq px_k(T)$ for every spanning tree $T$ of $G$.
\end{pro}

\begin{pro}\cite{Chen1}\label{pro2}
For an arbitrary connected graph $G$ with order $n\geq 3$, we have $px_k(G)\geq 2$ for any integer $k$ with $3\leq k\leq n$.
\end{pro}

A \emph{Hamiltonian path} in a graph $G$ is a path containing every vertex of $G$ and a graph having a Hamiltonian path is a \emph{traceable graph}.

\begin{pro}\cite{Chen1}\label{pro3}
If $G$ is a traceable graph with $n\geq 3$ vertices, then $px_k(G)=2$ for each integer $k$ with $3\leq k\leq n$.
\end{pro}

Armed with Proposition~\ref{pro3}, we can easily get
$px_k(K_n)=px_k(P_n)=px_k(C_n)=px_k(W_n)=px_k(K_{s,s})=2$ for each integer $k$ with $3\leq k\leq n$, where $K_n$, $P_n$, $C_n$ and $W_n$ are respectively a complete graph, a path, a cycle and a wheel on $n\geq 3$ vertices and $K_{s,s}$ is a regular complete bipartite graph with $s\geq 2$.

A vertex set $D\subseteq G$ is called an \emph{$s$-dominating set} of $G$ if every vertex in $\overline{D}$ is adjacent to at least $s$  distinct vertices of $D$.
If, in addition, $G[D]$ is connected, then we call $D$
a \emph{connected $s$-dominating set}.
Recently, Chang et al.~\cite{H.} gave an upper bound for the 3-proper index of graphs with respect to the connected 3-dominating set.

\begin{thm}\cite{H.}\label{thm1}
If $D$ is a connected 3-dominating set of a connected graph $G$ with minimum degree $\delta(G)\geq 3$, then $px_3(G)\leq px_3(G[D])+1$.
\end{thm}

Using this, we can easily get the following. 

\begin{thm}\label{thm2}
For any complete bipartite graph $K_{s,t}$ with $t\geq s\geq3$, we have $2 \leq px_3(K_{s,t})\leq 3$.
\end{thm}

\begin{proof}
Let $U$ and $W$ be the two partite sets of $K_{s,t}$, where $U=\{u_1,u_2,u_3,\dots,u_s\}$ and $W=\{w_1,w_2,w_3,\dots,w_t\}$. Obviously, $D=\{u_1,u_2,u_3,w_1, w_2,w_3\}$ is a connected 3-dominating set of $K_{s,t}$ and $\delta(K_{s,t})\geq 3$. It follows from Theorem~\ref{thm1} that $px_3(K_{s,t})\leq px_3(G[D])+1=3$.
By Proposition~\ref{pro2}, we have $px_3(K_{s,t})\geq 2$.
\end{proof}

Naturally, we wonder among them, whose 3-proper index is~$2$. Moreover, what are the exact values of $px_3(K_{s,t})$ with $s+t\geq 3, t\geq s\geq 1$ and $px_3(K_{n_1,n_2,\dots,n_r})$ with $r\geq3$?
Moreover, what happens when $k\geq 4$?
So our paper is organised as follows: In Section 2, we concentrate on all complete bipartite graphs and determine the value of the 3-proper index of each of them. In Section 3, we go on investigating all complete multipartite graphs and obtain the 3-proper index of each of them. In the final section, we turn to the case that $k\geq 4$, and give a partial answer. In the sequel, we use $c(uw)$ to denote the color of the edge $uw$.

\section{The $3$-proper index of a complete bipartite graph}
In this section, we concentrate on all complete bipartite graphs $K_{s,t}$ with $s+t\geq 3,t\geq s\geq 1$ and get a complete answer of the value of $px_3(K_{s,t})$. From~\cite{Chen1}, we know $px_3(K_{1,t})=t$.
Hence, in the following we assume that $t\geq s\geq 2$. Our result will be divided into three
separate theorems depending upon the value of $s$.


\begin{thm}\label{thm3}
For any integer $t\geq 2$, we have
\begin{displaymath}
px_3(K_{2,t}) = \left\{\begin{array}{ll}
2  & \text{if $2\leq t\leq 4$;}\\
3  & \text{if $5\leq t\leq 18$;}\\
\lceil\sqrt{\frac{t}{2}}\rceil  & \text{if $t\geq 19$.}
\end{array}\right.
\end{displaymath}
\end{thm}

{\noindent\bf Proof.}
Let $U,W$ be the two partite sets of $K_{2,t}$, where $U=\{u_1,u_2\}$ and $W=\{w_1,$ $w_2,\dots,w_t\}$. Suppose that there exists a 3-proper coloring $c:E(K_{2,t})\rightarrow \{1,2,\dots,k\}$,\\$k\in\mathbb{N}$.
Corresponding to the 3-proper coloring, there is a color code($w$) assigned to every vertex $w\in W$, consisting of an ordered 2-tuple $(a_1,a_2)$, where $a_i=c(u_iw)\in\{1,2,\dots,k\}$ for $i=1,2.$ In turn, if we give each vertex of $W$ a code, then we can induce the corresponding edge-coloring of $G$.

{\bf Claim 1:} $px_3(K_{2,t})=2$ if $2\leq t\leq4$.

\begin{proof} Give the codes $(1,2),(2,1),(1,1),(2,2)$ to $w_1,w_2,w_3,w_4$ (if there is).
Then it is easy to check that for every 3-subset $S$ of $K_{2,t}$, the edge-colored $K_{2,t}$ has a proper path $P$ connecting $S$.
\end{proof}

{\bf Claim 2:} $px_3(K_{2,t})>2$ if $t>4$.
\begin{proof}
Otherwise, give $K_{2,t}$ a 3-proper coloring with colors 1 and 2.
Then for any 3-subset $S$ of $K_{2,t}$, any proper tree connecting $S$ must be a path, actually.
For $t>4$, there are at least two vertices $w_p,w_q$ in $W$ such that $code(w_p)=code(w_q)$. We may assume that
$code(w_1)=code(w_2)$.
Then for an arbitrary integer $i$ with $3\leq i\leq t$, let $S=\{w_1, w_2, w_i\}$. There must be a proper path of length 4 connecting $S$.
Suppose that the path is $w_au_{a'}w_bu_{b'}w_c$, where $\{w_a, w_b, w_c\}=\{w_1, w_2, w_i\}$ and
$\{u_{a'},u_{b'}\}=\{u_1,u_2\}$. By symmetry, we can assume that $u_{a'}=u_1, u_{b'}=u_2$.
Then $w_b=w_i$
for otherwise we have $c(w_au_1)=c(u_1w_b)$
or $c(w_bu_2)=c(u_2w_c)$, a contradiction.
For equivalence, let $w_a=w_1, w_c=w_2$.
Thus $c(w_iu_1)\neq c(w_iu_2)$. Without loss of generality, we can suppose that $c(w_iu_1)=1$
and $c(w_iu_2)=2$. Hence, $code(w_i)=(1,2)$ for each integer $3\leq i\leq t$. Now let $S=\{w_3, w_4, w_5\}$. It is easy to verify that there is no proper path $w_au_{a'}w_bu_{b'}w_c$ connecting $S$, for
we always have $c(w_au_{a'})=c(u_{a'}w_b),c(w_bu_{b'})=c(u_{b'}w_c)$.
\end{proof}

{\bf Claim 3:} Let $k$ be a integer where $k\geq3$. Then $px_3(K_{2,t})\leq k$ for $4<t\leq2k^2$.

\begin{proof}
Set $code(w_1)=(1,1), code(w_2)=(1,2), \dots, code(w_k)=(1,k)$;

$code(w_{k+1})=(2,1), code(w_{k+2})=(2,2), \dots, code(w_{2k})=(2,k)$;

\dots

$code(w_{k(k-1)+1})=(k,1), code(w_{k(k-1)+2})=(k,2), \dots, code(w_{k^2})=(k,k)$ \\(if there is).
And let $code(w_{k^2+i})=code(w_i)$ for $1\leq i\leq k^2$ (if there is).
Now, we prove that this is a 3-proper coloring of $K_{2,t}$. First of all, we notice that each code appears at most twice. Let $S$ be a 3-subset of $K_{2,t}$. We consider the following two cases.

{\bf Case 1:} Let $S=\{w_l, w_m, w_n\}$, where $1\leq l< m< n\leq t$.

{\bf Subcase 1.1:} If there is a $j\in\{1,2\}$ such that the colors of $u_jw_l, u_jw_m, u_jw_n$ are
pairwise distinct, then the tree $T=\{u_jw_l, u_jw_m, u_jw_n\}$ is a proper $S$-tree.

{\bf Subcase 1.2:} If there is no such $j$, that is, at least two of the edges ${u_jw_l, u_jw_m, u_jw_n}$ share the same color for both $j=1$ and $j=2$.

$i)$ $code(w_l),code(w_m)$ and $code(w_n)$ are pairwise distinct.
Without loss of generality, we suppose that
$c(u_1w_l)=c(u_1w_m)=a,c(u_2w_l)=c(u_2v_n)=b$ ($1\leq a,b\leq k^2$). Then $c(u_1w_n)\neq c(u_1w_l),c(u_2w_l)\neq c(u_2w_m)$.
If $a=b$, then we have $c(u_1w_n)\neq c(w_nu_2)$. So the path $P=w_lu_1w_nu_2w_m$ is a proper $S$-tree.
Otherwise, the path $P=w_nu_1w_lu_2w_m$ is a proper $S$-tree.

$ii)$ Two of the codes of the vertices in $S$ are the same.
Without loss of generality, we assume that $code(w_l)=code(w_m)=(a,b), code(w_n)=(x,y)$
($1\leq a,b,x,y\leq k^2$).
Notice that $(x,y)\neq(a,b)$, then suppose that $x\neq a$. Since $k\geq3$, there are two positive integers $p,q\leq k$
such that $p\neq a,p\neq x$ and $q\neq b,q\neq p$. Pick a vertex $w_r$ whose code is $(p,q)$ (this vertex exists since all of the $k^2$ codes appear at least once). Then the tree $T=\{u_1w_m, u_1w_n, u_1w_r, w_ru_2, u_2w_l\}$
is a proper $S$-tree.

{\bf Case 2:} $S=\{u_r,w_l,w_m\}$, where $1\leq l< m\leq t$. By symmetry, let $r=1$.

Suppose that $code(w_l)=(a,b), code(w_m)=(x,y)$
($1\leq a,b,x,y\leq k^2$).
If $a\neq x$ then the path $P=w_lu_1w_m$ is a proper $S$-tree.
If $a=x$, then we consider whether $b=y$ or not. We discuss two subcases.

$i)$ $b\neq y$, then at least one of them is not equal to $a$, assume that $b\neq a$. So the path $P=u_1w_lu_2w_m$ is a proper $S$-tree.

$ii)$  $b=y$, that is $code(w_l)=code(w_m)$, so all of the $k^2$ codes appear at least at once. Since $k\geq3$, there are two positive integers $p,q\leq k$
such that $p\neq a$ and $q\neq b,q\neq p$. Pick a vertex $w_r$ whose code is $(p,q)$. Then the path $P=w_lu_1w_ru_2w_m$
is a proper $S$-tree.

{\bf Case 3:} $S=\{u_1,u_2,w_l\}$, where $1\leq l\leq t$.

Suppose that $code(w_l)=(a,b)$
($1\leq a,b\leq k^2$). If $a\neq b$, then the path $P=u_1w_lu_2$ is a proper $S$-tree. Otherwise, according to our edge-coloring, there exists a vertex $w_r$ of $W$ with the code $(p,q)$ such that $q\neq a$ and $p\neq q$. Then the path $P=w_lu_2w_ru_1$ is a proper $S$-tree.
\end{proof}

{\bf Claim 4:} $px_3(K_{2,t})> k$ for $t>2k^2$.

\begin{proof}
For any edge-coloring of $K_{2,t}$ with $k$ colors, there must be a code which appears at least
three times. Suppose that $w_1, w_2, w_3$ are the vertices with the same code
and set $S=\{w_1, w_2, w_3\}$.
Then for any tree $T$ connecting $S$, there is a $j\in \{1,2\}$ such that $\{u_jw_l, u_jw_m\}\subseteq E(T)$ for
some $\{l,m\}\subseteq\{1, 2, 3\}$, $l\neq m$.
But $c(u_jw_l)=c(u_jw_m)$, so $T$ can not be a proper $S$-tree. Thus $px_3(K_{2,t})>k$.
\end{proof}

By Claims 2-4, we have the following result:
if $5\leq t\leq8$, $px_3(K_{2,t})=3$;
if $t>8$, let $k=\lceil\sqrt{\frac{t}{2}} \rceil$, then $3\leq \sqrt{\frac{t}{2}}\leq k<\sqrt{\frac{t}{2}}+1$,
i.e., $2(k-1)^2+1\leq t\leq 2k^2$, so
we have $px_3(K_{2,t})=k=\lceil\sqrt{\frac{t}{2}}\rceil$.
Notice that $px_3(K_{2,t})=3$ for $5\leq t\leq 18$.\hspace{3.5cm}{$\blacksquare$}


\begin{thm}\label{thm4}
For any integer $t\geq 3$, we have
\begin{displaymath}
px_3(K_{3,t}) = \left\{\begin{array}{ll}
2  & \text{if $3\leq t\leq 12$;}\\
3  & \text{otherwise.}
\end{array}\right.
\end{displaymath}
\end{thm}

{\noindent\bf Proof.}
Let $U,W$ be the two partite sets of $K_{3,t}$, where $U=\{u_1,u_2,u_3\}$ and $W=\{w_1,w_2,\dots,w_t\}$.
Suppose that there exists a 3-proper coloring $c:E(K_{2,t})\rightarrow \{0,1,2,\dots,$ $ k-1\}$, $k\in\mathbb{N}$. Analogously to Theorem~\ref{thm3}, corresponding to the 3-proper coloring, there is a color code($w$) assigned to every vertex $w\in W$, consisting of an ordered 3-tuple $(a_1,a_2,a_3)$, where $a_i=c(u_iw)\in\{0,1,2,\dots,k-1\}$ for $i=1,2,3.$ In turn, if we give each vertex of $W$ a code, then we can induce the corresponding edge-coloring of $G$.

{\bf Case 1:} $3\leq t\leq 8$.

In this part, we give the vertices of $W$ the codes which induce a 3-proper coloring of $K_{3,t}$
with colors 0 and 1.
And by application of binary system, we can introduce the assignment of the codes in a clear way.
Recall the Abelian group  $\mathbb{Z}_2$.
We build a bijection
$f:\{w_1,w_2,\dots,w_8\}\rightarrow \mathbb{Z}_2\times \mathbb{Z}_2\times \mathbb{Z}_2$, where
$f(w_{4a_1+2a_2+a_3+1})= (a_1,a_2,a_3)$.
For instance, $f(w_3)=(0,1,0)$.
Under this condition, we use its restriction $f_W$ on $W$.
Now, we prove that $f$ induces a 3-proper coloring of $K_{3,t}$.
Let $S$ be an arbitrary 3-subset.

{\bf Subcase 1.1:} $S=\{w_l,w_m,w_n\}$.

Because there is no copy of any code, we can find a vertex in $U$, say $u_1$, such that
$u_1w_l,u_1w_m,u_1w_n$ are not all with the same color.
We may assume that $c(u_1w_l)=c(u_1w_m)=0$ and $c(u_1w_n)=1$.

$i)$ $code(w_l)=(0,0,0)$. Then there is a `1' in the code of $w_m$. By symmetry, assume
that $c(u_2w_m)=1$. Then there is a proper path $P=w_lu_2w_mu_1w_n$ connecting $S$.

$ii)$ $code(w_l)=(0,0,1)$. If $code(w_m)=(0,0,0)$, then we return to $i)$. Otherwise, the code of
$w_m$ is neither $(0,0,0)$ nor $(0,0,1)$. So $c(u_2w_m)=1$.
Then the proper $S$-tree is the same as that in $i)$.

$iii)$ $code(w_l)=(0,1,0)$. It is similar to $ii)$.

$iv)$ $code(w_l)=(0,1,1)$. Then either $c(u_2w_m)=0$ or $c(u_3w_m)=0$. By symmetry, we suppose that $c(u_2w_m)=0$. Then the path $P=w_mu_2w_lu_1w_n$ is a proper $S$-tree.

{\bf Subcase 1.2:} $S=\{u_j,w_l,w_m\}$.

If $c(u_jw_l)\neq c(u_jw_m)$, then the path $P=w_lu_jw_m$ is a proper $S$-tree. Otherwise, by symmetry, we assume that $c(u_jw_l)= c(u_jw_m)=0$, then there is a $j'\neq j$ such that
$c(u_{j'}w_l)\neq c(u_{j'}w_m)$ (otherwise $w_l,w_m$ will have the same code). So one of $c(u_{j'}w_l)$ and $c(u_{j'}w_m)$ equals 1, say $c(u_j'w_l)=1$.
Then the path $P=u_jw_lu_{j'}w_m$ is a proper $S$-tree.

{\bf Subcase 1.3:} $S=\{u_{j_1},u_{j_2},w_l\}$.

If $c(u_{j_1}w_l)\neq c(u_{j_2}w_l)$, then the path $P=u_{j_1}w_lu_{i_2}$ is a proper $S$-tree.
Otherwise, by symmetry, we assume that $c(u_{j_1}w_l)= c(u_{j_2}w_l)=0$. By the sequence of the codes according to $f$ and $t\geq3$,
we know that for any two vertices $u_{a'},u_{b'}$ of $U$, there exists a vertex $w\in W$ such that $c(u_{a'}w)\neq c(u_{b'}w)$.
Similar to Subcase 1.2, we can obtain a proper
$S$-tree.

{\bf Subcase 1.4:} $S=\{u_1,u_2,u_3\}$.

$P=u_1w_3u_2w_2u_3$ is a proper path connecting $S$.

{\bf Case 2:} $9\leq t\leq 12$.

Set
$code(w_1)=(0,0,1), code(w_2)=(0,1,0), code(w_3)=(0,1,1)$,

$code(w_4)=(1,0,0), code(w_5)=(1,0,1), code(w_6)=(1,1,0)$.

And let $code(w_{6+i})=code(w_i)$ for $1\leq i\leq6$(if there is).
For convenience, we denote $w_{6+i}=w_i'$. Now, we claim that this induces a 3-proper coloring of $K_{3,t}$.
Let $S$ be an arbitrary 3-subset of $K_{3,t}$.
Based on Case 1, we only consider about the case that $\{w_i,w_i'\}\subseteq S$
for some $1\leq i\leq6$. By symmetry, we suppose that $i=1$.
First of all, we list three proper paths containing ${w_1,w_1'}$: $P_1=w_1u_3w_2u_2w_1'$, $ P_2=w_1u_2w_3u_1w_4u_3
w_1'$ and $P_3=w_1u_1w_5u_2w_6u_3w_1'$, in which $w_j$ can be replaced by $w_j'$ for $2\leq j\leq6$.
Then, we can always find a proper path from $\{P_1,P_2,P_3\}$ connecting $S$ whichever the third vertex of $S$ is.

{\bf Case 3:} $t\geq13$.

Take into account Theorem~\ref{thm2}, we claim that $px_3(K_{3,t})=3$.
We prove it by contradiction.
If there is a 3-proper coloring of $K_{3,t}$ with two colors 0 and 1, then any proper tree for an arbitrary 3-subset $S$ is in fact a path.
Consider about the set $S\subseteq W$. As
the graph is bipartite and we just care about the shortest proper path connecting $S$,
there are only two possible types of such a path:

\uppercase\expandafter{\romannumeral1}: $w_au_{a'}w_bu_{b'}w_c$

\uppercase\expandafter{\romannumeral2}: $w_au_{a'}w_bu_{b'}w'u_{c'}w_c$\\
where $\{u_{a'},u_{b'},u_{c'}\}=U$ and $\{w_a,w_b,w_c\}=S, w'\in W \setminus S$.

Firstly, as $t\geq13$, we know that some code appears more than once.
But it can not appear more than twice. Otherwise, suppose that $w_i,w_i',w_i''$ are the three vertices with the same code, and let
$S=\{w_i,w_i',w_i''\}$.
Whether the proper path connecting $S$ is type
\uppercase\expandafter{\romannumeral1} or type \uppercase\expandafter{\romannumeral2}, it should
be $c(w_au_{a'})\neq c(w_bu_{a'})$, contradicting with the same code of the three vertices.

Secondly, we prove the following several claims by contradiction.

{\bf Claim 1:} The repetitive code can not be $(0,0,0)$ or $(1,1,1)$.

\begin{proof}
Suppose that $code(w_1)=code(w_2)=(0,0,0)$. Let $S=\{w_1,w_2,w_3\}$ where $w_3\in W\backslash \{w_1,w_2\}$, and let $P$ be a proper path connecting $S$. Then $w_1,w_2$ are the two end vertices of $P$, and so the two end edges of it are assigned the same color. However, since the length of $P$ is even, the colors of the end edges can not be the same, a contradiction. Analogously, the code $(1,1,1)$ cannot appear more than once.
\end{proof}

{\bf Claim 2:} If the code $(0,0,1)$ is repeated, then there is no vertex in $W$ with $(0,0,0)$ as its code.
\begin{proof}
Suppose that $code(w_1)=code(w_2)=(0,0,1), code(w_3)=(0,0,0)$. Let $S=\{w_1,w_2,w_3\}$, and let $P$ be a proper path connecting $S$. Then $w_3$ is one of the end vertices of $P$.
Moreover, the path $P$ must be type \uppercase\expandafter{\romannumeral2},
for in type \uppercase\expandafter{\romannumeral1}, we need
$c(w_au_{a'})\neq c(w_bu_{a'})$ and $c(w_bu_{b'})\neq c(w_cu_{b'})$, which is impossible for $S$.
We can also deduce that $u_{a'}=u_3$ because $c(w_au_{a'})\neq c(w_bu_{a'})$.
And $\{w_1,w_2\}\neq \{w_a,w_b\}$ since they are with the same code. So we have $w_a=w_3$.
Thus, $\{w_b,w_c\}=\{w_1,w_2\}$ and $\{u_{b'},u_{c'}\}=\{u_1,u_2\}$, contradicting with the fact
that $c(w_bu_{b'})\neq c(w_cu_{c'})$.
\end{proof}

Analogously, we have that the repetitive code $(0,1,0)$ or $(1,0,0)$ can not exist along with the code $(0,0,0)$,
respectively.  And the repetitive code $(0,1,1),(1,0,1)$ or $(1,1,0)$ can not exist along with the code $(1,1,1)$,
respectively.

Finally, as $t\geq13$ and no code could appear more than twice, there are at least~$7$ different codes in $W$
and at least~$5$ codes repeated. But considering about Claim~$2$ and its analogous results, it is a contradiction.
So $px_3(K_{3,t})=3$ when $t\geq13$.\hspace{2cm}{$\blacksquare$}


\begin{thm}\label{thm5}
For a complete bipartite graph $K_{s,t}$ with $t\geq s\geq4$, we have $px_3(K_{s,t})=2$.
\end{thm}

{\noindent\bf Proof.} Let $U,W$ be the two partite sets of $K_{s,t}$, where $U=\{u_1,u_2,\dots,u_s\}$ and $W=\{w_1,w_2,\dots,w_t\}$.
And denote a cycle $C_s=u_1w_1u_2w_2\dots u_sw_su_1$.
Moreover, if $u,v\in V(C_s)$, then we use $uC_sv$ to denote the segment of $C_s$ from $u$ to $v$ in the clockwise direction, otherwise we denote it by $uC_s'v$.
Then we demonstrate a 3-proper coloring of $K_{s,t}$ with two colors 0 and 1.
Let $c(u_iw_i)=0$ ($1\leq i\leq s$) and $c(u_iw_j)=1$ ($1\leq i\neq j\leq s$).
And assign $c(w_ru_i)=i\ (mod\ 2)$ ($1\leq i \leq s,s<r\leq t$).
Now we prove that this coloring is a 3-proper coloring of $K_{s,t}$. Consider about its 3-subset $S$.

$i)$ $S\subseteq V(C_s)$. The proper path is in $C_s$.

$ii)$ $S=\{w_l,w_m,w_n\}$ where $l,m,n>s$.
Then the path $P=w_lu_1w_1u_2w_mu_3w_3u_4w_n$ is a proper $S$-tree.

$iii)$ $S=\{w_l,w_m,w_n\}$ where $l\leq s, m,n>s$.
If $c(w_mu_l)=1$, then the path $P=w_mu_lw_lC_su_2w_n$ is a proper $S$-tree.
If $c(w_mu_l)=0$, then the proper $S$-tree is the path
$P=w_mu_lw_{l-1}u_{l-1}w_nu_{l-2}C_s'w_l$, where $u_0=u_s,u_{-1}=u_{s-1}$ if $i_1= 2$.

$iv)$ $S=\{u_j,w_l,w_m\}$ where $l,m>s$.
The way to find a proper $S$-tree is similar with that in $iii)$.

$v)$ $S=\{u_j,w_l,w_m\}$ where $l\leq s,m>s$. If $c(w_mu_j)=1$, then the proper $S$-tree is the path $P=w_mu_jw_jC_sw_l$.
If $c(w_mu_j)=0$, then the path $P=w_mu_jC_s'w_l$ is a proper $S$-tree.

$vi)$ $S=\{u_{j_1},u_{j_2},w_i\}$ where $i>s$.
The way to find a proper $S$-tree is similar with that in $v)$.\hspace{13.4cm}{$\blacksquare$}

{\noindent\bf Remarks.}
Here, we introduce a generalization of $k$-proper index which was proposed by Chang et. al. in~\cite{H.1} recently. Let $G$ be a nontrivial $\kappa$-connected graph of order $n$, and let $k$ and $\ell$ be two integers with $2\leq k\leq n$ and $1\leq \ell \leq \kappa$. For $S\subseteq V(G)$, let $\{T_1,T_2,\dots,T_\ell\}$ be a set of $S$-tree, they are
\emph{internally disjoint} if $E(T_i)\cap E(T_j)=\emptyset$ and $V(T_i)\cap V(T_j)=S$ for every pair of distinct integers $i,j$ with $1\leq i,j \leq  \ell$. The $(k,\ell)$-proper index of $G$, denoted by $px_{k,\ell}(G)$, is the minimum number of colors that are required in an edge-coloring of $G$ such that for every $k$-subset $S$ of $V(G)$, there exist $\ell$ internally disjoint proper $S$-trees connecting them. In their paper, they investigated the complete bipartite graphs and obtained the following.

\begin{thm}\cite{H.1}\label{thm6}
 Let $s$ and $t$ be two positive integers with $t=O(s^r), r\in\mathbb{R}$ and $r\geq 1$. For every pair of integers $k, \ell$ with $k\geq 3$, there exists a positive integer $N_3=N_3(k,\ell)$ such that $px_{k,\ell}(K_{s,t})=2$ for every integer $s\geq N_3$.
\end{thm}

Obviously, they did not give the exact value of $px_{k,\ell}(K_{s,t})$, even for $k=3$ and $\ell=1$.
Our Theorem~\ref{thm5} completely determines the value of $px_{k,\ell}(K_{s,t})$ for $k=3$ and $\ell=1$, without
using the condition that $t=O(s^r), r\in\mathbb{R}$ and $r\geq 1$.

\section{The $3$-proper index of a complete multipartite graph}
With the aids of Theorems~\ref{thm3},~\ref{thm4} and~\ref{thm5}, we are now able to determine the 3-proper index of all complete multipartite graphs. First of all, we give a useful theorem.

\begin{thm}\cite{Dirac}\label{thm7}
Let $G$ be a graph with $n$ vertices. If $\delta(G)\geq \frac{n-1}{2}$, then $G$ has a Hamiltonian path.
\end{thm}

\begin{thm}\label{thm8}
Let $G=K_{n_1,n_2,\dots,n_r}$ be a complete multipartite graph, where $r\geq3$ and $n_1\leq n_2\leq \dots\leq n_r$. Set $s=\sum^{r-1}_{i=1}n_i$ and $t=n_r$. Then we have
\begin{displaymath}
px_3(G) = \left\{\begin{array}{ll}
3  & \text{if $G=K_{1,1,t}, 5\leq t\leq 18$}\\
   &  \text{  \ or $G=K_{1,2,t}, t\geq 13$}\\
   &  \text{  \ or $G=K_{1,1,1,t}, t\geq 15$;}\\
\lceil\sqrt{\frac{t}{2}}\rceil  & \text{if $G=K_{1,1,t}, t\geq 19$;}\\
2  & \text{otherwise.}
\end{array}\right.
\end{displaymath}
\end{thm}

{\noindent\bf Proof.} The graph $G$ has a $K_{s,t}$ as its spanning subgraph, so it follows from Propositions~\ref{pro1} and~\ref{pro2} that $2\leq px_3(G)\leq px_3(K_{s,t})$. In the following, we discuss two cases according to the relationship between $s$ and $t$.

{\bf Case 1:} $s\leq t$. Let $U_1,U_2,\dots,U_r$ denote the different $r$-partite sets of $G$, where $|U_i|=n_i$ for each integer $1\leq i\leq r$.

When $s\geq 4$, then by Theorem~\ref{thm5}, we have $px_3(G)= px_3(K_{s,t})=2$.
When $s\leq 3$, there are only three possible values of $(n_1,n_2,\dots,n_{r-1})$.

{\bf Subcase 1:} $(n_1,n_2,\dots,n_{r-1})=(1,1)$. Set $U_1=\{u_1\}, U_2=\{u_2\}$. Under this condition, giving the edge $u_1u_2$ an arbitrary color, the proof is exactly the same as that of Theorem \ref{thm3}. So it holds that $px_3(G)=px_3(K_{2,t})$.

{\bf Subcase 2:} $(n_1,n_2,\dots,n_{r-1})=(1,2)$. Set $U_1=\{u_1\}, U_2=\{u_2,u_3\}$ and $W=U_r$. By Theorem~\ref{thm4}, we have $px_3(G)=px_3(K_{3,t})=2$ if $t\leq 12$;
$px_3(G)\leq px_3(K_{3,t})=3$ if $t>12$. We claim that $px_3(G)=3$ if $t>12$. Assume, to the contrary, that $G$ has a 3-proper coloring with
two colors 0 and 1.
By symmetry, we assume that $c(u_1u_2)= 0$.
With the similar reason in Case~$3$ of the proof of Theorem~\ref{thm4}, no code can appear more than twice.
 And recall the bijection $f$ defined in that proof. To label the vertices in $W$, we use its inverse $f^{-1}:
 (a_1,a_2,a_3)\mapsto w_{4a_1+2a_2+a_3+1}$, and denote by $w_i'$ the copy of the vertex $w_i$ with $1\leq i\leq 8$. Then we prove the following results by contradiction.

{\bf Claim 1:} $\{w_1,w_1',w_2\}\nsubseteq W$ and $\{w_2,w_2',w_1\}\nsubseteq W$.

\begin{proof}
Set $S=\{w_1,w_1',w_2\}$. We know from the proof of Theorem~\ref{thm4} that there is no proper path of type \uppercase\expandafter{\romannumeral1} or \uppercase\expandafter{\romannumeral2}.
So the proper path $P$ connecting $S$ is type \uppercase\expandafter{\romannumeral3}: $w_au_{a'}w_bu_{b'}u_{c'}w_c$.
Then $w_1,w_1'$ must be the end vertices of $P$, and so $w_b=w_2$ and $u_{a'}=u_3$. Since $c(w_au_{a'})= 0$, $c(u_{b'}u_{c'})= 1$, contradicting with $c(u_1u_2)= 0$. Hence, we get $\{w_1,w_1',w_2\}\nsubseteq W$.
Similarly, we have $\{w_2,w_2',w_1\}\nsubseteq W$.
\end{proof}

{\bf Claim 2:} $\{w_4,w_4',w_8\}\nsubseteq W$
and $\{w_8,w_8',w_4\}\nsubseteq W$.

\begin{proof}
 Set $S=\{w_4,w_4',w_8\}$. Similar to Claim~$1$,
any proper path $P$ connecting $S$ should be type \uppercase\expandafter{\romannumeral3}: $w_au_{a'}w_bu_{b'}u_{c'}w_c$.
Then $w_8$ must be an end vertex of $P$, and so both of the end edges of $P$ are colored with 1. Thus $u_{a'}=u_1$. Then $\{u_{b'},u_{c'}\}=\{u_2,u_3\}$ and $c(u_2u_3)$=0, contradicting with the fact that $u_2u_3\notin E(G)$.
Similarly, we have $\{w_8,w_8',w_4\}\nsubseteq W$.
\end{proof}

So there are four cases that some vertices can not exist in $W$ at the same time, and each code appears at most twice. However, there are more than 12 vertices in $W$, a contradiction.
So $px_3(G)=px_3(K_{3,t})=3$ when $t>12$.

{\bf Subcase 3:} $(n_1,n_2,\dots,n_{r-1})=(1,1,1)$. Set $U=\cup_{j=1}^{r-1}U_j=\{u_1,u_2,u_3\}$ and $W=U_r$.

{\bf Claim 3:} $px_3(G)=2$ if $t\leq 14$.

\begin{proof} By Theorem~\ref{thm4}, we have $px_3(G)=px_3(K_{3,t})=2$ if $t\leq 12$;
$px_3(G)\leq px_3(K_{3,t})=3$ if $t>12$.
When $t=13$ or 14, we recall $code(w)$ defined in Case 2 of Theorem~\ref{thm4}.
Set

$code(w_1)=(0,0,1), code(w_2)=(0,1,0), code(w_3)=(0,1,1),code(w_4)=(1,0,0)$,

$code(w_5)=(1,0,1), code(w_6)=(1,1,0), code(w_7)=(1,1,1)$.

And let $code(w_{7+i})=code(w_i)$ for $1\leq i\leq7$ (if there is) and $c(u_iu_j)=0$ for $1\leq i\neq j\leq 3$. For convenience, we denote $w_{7+i}=w_i'$. Now, we claim that this induces a 3-proper coloring of $G$.
Let $S$ be an arbitrary 3-subset of $G$.
Based on Theorem~\ref{thm4}, we only consider about the case that $w_7(w_7')\in S$.
When $S=\{w_1,w_7,w_7'\}$,
then the path $P=w_7u_1w_1u_3u_2w_7'$ is a proper path connecting $S$. Similarly, we can find a proper path in type \uppercase\expandafter{\romannumeral3} connecting $S$ whichever the two other vertices of $S$ are.
\end{proof}

{\bf Claim 4:} $px_3(G)=3$ if $t>14$.

\begin{proof} Assume, to the contrary, that $G$ has a 3-proper coloring with
two colors 0 and 1. If the edges of $G[U]$ are colored with two different colors, then we set $u_2$ the common vertex of two edges with two different colors.
Moreover, without loss of generality, we suppose that $c(u_1u_2)=0$.
Similar to Subcase 2, we have $px_3(G)=3$ if $t>12$.
If all the edges of $G[U]$ are colored with one color, say 0.
Repeat the discussion in Subcase~$2$, then we know Claim~$1$ is also true under this condition.
As $t\geq15$ and no code could appear more than twice, there are at least 8 different codes in $W$
and at least 7 codes repeated. But from Claim 1, we know $\{w_1,w_1',w_2\}\nsubseteq W$ and $\{w_2,w_2',w_1\}\nsubseteq W$.
So $px_3(G)=3$ when $t\geq15$.
\end{proof}

{\bf Case 2:} $s\geq t$. Under this condition, we have $\delta(G)\geq \frac{n-1}{2}$. By Theorem~\ref{thm7}, we know $G$ is traceable. Thus, it follows from Proposition~\ref{pro3} that $px_3(G)=2$.\hspace{2cm}{$\blacksquare$}

\section{The $k$-proper index}
Now, we turn to the $k$-proper index of a complete bipartite graph and a complete multipartite graph for general $k$. Throughout this section, let $k$ be a fixed integer with $k\geq 3$. Firstly, we generalize Theorem~\ref{thm1} to the $k$-proper index.

\begin{thm}\label{thm9}
If $D$ is a connected $k$-dominating set of a connected graph $G$ with minimum degree $\delta(G)\geq k$, then $px_k(G)\leq px_k(G[D])+1$.
\end{thm}

\begin{proof} Since $D$ is a connected $k$-dominating set, every vertex $v$ in $\overline{D}$ has at least $k$ neighbors in $D$.
Let $x=px_k(G[D])$.
We first color the edges in $G[D]$ with $x$ different colors from $\{2,3,\dots,x+1\}$ such that
for every $k$ vertices in $D$, there exists a proper tree in $G[D]$ connecting them.
Then we color the remaining edges with color~$1$.

Next, we will show that this coloring makes $G$ $k$-proper connected.
Let $S=\{v_1,v_2,\dots,v_k\}$ be any set of $k$ vertices in $G$.
Without loss of generality, we assume that $\{v_1,\dots,v_p\}\subseteq D$ and $\{v_{p+1},\dots,v_k\}\subseteq \overline{D}$ for some $p$ $(0\leq p\leq k)$.
For each $v_i\in \overline{D}$ $(p+1\leq i\leq k)$,
let $u_i$ be the neighbour of $v_i$ in $D$ such that $\{u_{p+1},\dots,u_k\}$ is a $(k-p)$-set.
It is possible since $D$ is a $k$-dominating set.
Then the edges $\{u_{p+1}v_{p+1},\dots, u_kv_k\}$ together with the proper tree connecting the vertices $\{v_1,\dots,v_p,u_{p+1},\dots,u_k\}$ in $G[D]$ induces a proper $S$-tree.
Thus, we have $px_k(G)\leq px_k(G[D])+1$.
\end{proof}

Based on this theorem, we can give a lower bound and a upper bound on the $k$-proper index of a complete bipartite graph, whose proof is similar to Theorem~\ref{thm2}.

\begin{thm}\label{thm10}
For a complete bipartite graph $K_{s,t}$ with $t\geq s\geq k$, we have $2 \leq px_k(K_{s,t})\leq 3$.
\end{thm}

Let $G$ be a complete bipartite graph. Using the techniques in Theorem~\ref{thm5}, we can obtain the sufficient condition such that $px_k(G)=2$.

\begin{thm}\label{thm11}
For a complete bipartite graph $K_{s,t}$ with $t\geq s\geq 2(k-1)$, we have $px_k(K_{s,t})=2$.
\end{thm}

\begin{proof}
We demonstrate a $k$-proper coloring of $K_{s,t}$ with two colors 0 and 1, the same as Theorem~\ref{thm5}. For completeness, we restate the coloring. Let $U,W$ be the two partite sets of $K_{s,t}$, where $U=\{u_1,u_2,\dots,u_s\}$ and $W=\{w_1,w_2,\dots,w_t\}$, $t\geq s\geq 2(k-1)$.
Denote a cycle $C_s=u_1w_1u_2w_2\dots u_sw_su_1$.
Let $c(u_iw_i)=0$ ($1\leq i\leq s$) and $c(u_iw_j)=1$ ($1\leq i\neq j\leq s$).
And assign $c(w_ru_i)=i\ (mod\ 2)$ ($1\leq i \leq s,s<r\leq t$).
Now, we show that for any $k$-subset $S\subseteq V(K_{s,t})$,
there is a proper path $P_S$ connecting all the vertices in $S$.
Set $W_1=\{w_1,w_2,\dots,w_s\}$ and $W_2=\{w_{s+1},\dots,w_t\}$ (if $t>s$). Then $S$ can be divided into three
parts, i.e., $S=S_1\cup S_2\cup S_3$, where $S_1=S\cap W_1$, $S_2=S\cap W_2$ and $S_3=S\cap U$.
Suppose $|S_1|=p$, $|S_2|=q$, then $p+q\leq k$.
If $q=0$, the path $P=u_1w_1u_2w_2\dots u_sw_s$ is the proper path connecting $S$. If $q\geq 1$, set $S_2=\{w_{\alpha_1,}w_{\alpha_2},\dots,w_{\alpha_q}\}$,
where $s<\alpha_1,\alpha_2,\dots,\alpha_q\leq t$. Let $P=w_{\alpha_q}u_1w_1u_2w_2\dots u_sw_s$. Then consider the vertex set $W'_S=\{w_{2i}: w_{2i}\in W_1\setminus S_1\}$.
We have $|W'_S|\geq s/2-p\geq k-p-1\geq q-1$. So set $|W'_S|=\ell$ and $W'_S=\{w_{\beta_1},w_{\beta_2},\dots,w_{\beta_{q-1}},\dots,w_{\beta_\ell}\}$,
where $2\leq \beta_1,\beta_2,\dots,\beta_\ell \leq s$ are
even.
Then we construct a path $P_S$ by replacing the subpath $u_{\beta_j}w_{\beta_j}u_{\beta_j+1}$ of $P$ with $u_{\beta_j}w_{\alpha_j}u_{\beta_j+1}$ (and $u_sw_s $ with $u_sw_{\alpha_j}$ if $\beta_j=s$) for $1\leq j\leq q-1$. Hence, the new path $P_S$ is a proper path contains all the vertices of $U$ so that $P_S$ connects $S_3$. By the replacement we know that $P_S$ also connects $S_1$ as well as $S_2$. Thus we complete the proof.
\end{proof}

With the aids of Theorems~\ref{thm11} and~\ref{thm7}, we can easily get the following, whose proof is similar to Theorem~\ref{thm8}.

\begin{thm}\label{thm12}
Let $G=K_{n_1,n_2,\dots,n_r}$ be a complete multipartite graph, where $r\geq3$ and $n_1\leq n_2\leq \dots\leq n_r$. Set $s=\sum^{r-1}_{i=1}n_i$ and $t=n_r$. If $t\geq s\geq 2(k-1)$ or $t\leq s$, then we have
$px_k(G)=2$.
\end{thm}

\end{document}